\documentclass[a4paper,10pt,reqno]{amsart}

\usepackage[a4paper,width={16cm},left=2.5cm,bottom=3.5cm, top=3.5cm]{geometry}

\usepackage{color}
\usepackage{hyperref}
\usepackage[english]{babel} 
\usepackage[utf8]{inputenc} 
\usepackage[T1]{fontenc} 
\usepackage{amsmath}
\usepackage{amsfonts}
\usepackage{amssymb}
\usepackage{amsthm}
\usepackage{comment}
\usepackage{lmodern}
\usepackage{mathrsfs}
\usepackage{graphics}
\usepackage{pstricks}
\usepackage{pst-plot}
\usepackage{auto-pst-pdf}
\usepackage{enumerate}
\usepackage{pgf,tikz}
\usepackage{cases}
\usepackage{subfigure}
\usepackage{xfrac}

\numberwithin{equation}{section} 
\numberwithin{table}{section}

{ 
\theoremstyle{plain}
\newtheorem{theorem}{Theorem}[section]

\newtheorem{lemma}[theorem]{Lemma}

\newtheorem{remark}[theorem]{Remark}
\newtheorem{definition}[theorem]{Definition}

}

\def\R{\mathbb{R}}
\def\Z{\mathbb{Z}}
\def\C{\mathbb{C}}
\def\N{\mathbb{N}}

\begin{document}

\title[The first AB eigenvalue on the disk]{Multiple Aharonov--Bohm eigenvalues: the case of the first eigenvalue on the disk}

\author{Laura Abatangelo}
\address{Dipartimento di Matematica e Applicazioni, Università degli Studi di Milano-Bicocca,
\newline \indent  Via Cozzi 55, 20125 Milano, Italy.}
\email{laura.abatangelo@unimib.it}

\thanks{}

\date{\today}

\begin{abstract}
  It is known that the first eigenvalue for Aharonov--Bohm operators with half-integer circulation
  in the unit disk is double 
  if the potential's pole is located at the origin. 
  We prove that in fact it is simple as the pole $a\neq 0$. 
\end{abstract}



\maketitle

\section{Introduction}

In the present paper we are interested in the 
spectral properties of Schrödinger operators with Aharonov--Bohm 
vector potential (see e.g. \cite{AB,MOR,AdamiTeta1998}),
acting on functions $u \, : \, \R^2 \to \C$, i.e.
\begin{equation} \label{eq:operator}
(i\nabla + A_a^\alpha)^2 u := -\Delta u + 2 i A_a^\alpha \cdot 
\nabla u + |A_a^\alpha|^2 u,
\end{equation}
where the vector potential is singular at the point $a$ and takes the form
\begin{equation}
A_a^\alpha (x_1, x_2) = \alpha 
\left( - \frac{x_2 - a_2}{(x_1 - a_1)^2 + (x_2 - a_2)^2}, 
\frac{x_1 - a_1}{(x_1 - a_1)^2 + (x_2 - a_2)^2} \right).
\end{equation}
We address here its eigenvalues in the unit disk in the special case when circulation $\alpha=\tfrac12$.

In order to pose the problem, we address here the general functional setting. 
If $\Omega \subset \R^2$ is open, bounded and simply connected, for $a \in \Omega$, 
we define the functional space $H^{1,a}_0(\Omega,\C)$ as the completion of 
$C^\infty_c
(\Omega \setminus \{a\}, \C)$
with respect to the norm
\[
\| u \|_{H^{1,a}(\Omega,\C)} := \left( \| \nabla u \|_{L^2(\Omega,\C^2)}^2 
+ \| u \|_{L^2(\Omega,\C)}^2 + \left\| \frac{u}{|x-a|} 
\right\|_{L^2(\Omega,\C)}^2 \right)^{1/2}.
\]
When the circulation of the vector potential is not an integer, i.e.\ $\alpha \in \R \setminus \Z$, 
the latter norm is equivalent to the norm
\begin{equation*}
\| u \|_{H^{1,a}(\Omega,\C)} = \left( \left\| ( i \nabla + A_a^\alpha ) u 
\right\|^2_{ L^2(\Omega,\C) } + \| u \|^2_{ L^2(\Omega,\C) } \right)^{\!\!1/2},
\end{equation*}
by the Hardy type inequality proved in \cite{LaptevWeidl1999} (see also 
\cite{Balinsky} and
\cite[Lemma 3.1 and Remark 3.2]{FelliFerreroTerracini2011})
\[
\int_{ D_r(a) } | (i\nabla + A_a^\alpha) u |^2 \,dx \geq \Big( \min_{j \in \Z} 
|j - \alpha| \Big)^2 \int_{ D_r(a) } \frac{ |u(x)|^2 }{|x-a|^2} \, dx,
\]
which holds for all $r > 0$, $a \in \R^2$  and $u \in H^{1,a}(D_r(a),\C)$. Here we 
denote as $D_r(a)$ the disk of center $a$ and radius $r$.

By a Poincaré type inequality, see e.g.\ \cite[A.3]{AbatangeloFelliNorisNys2016}, we 
can consider the equivalent norm on $H^{1,a}_0(\Omega, \C)$
\[
\| u \|_{H^{1,a}_0(\Omega,\C)} := \left( \| (i \nabla + A_a^\alpha) u 
\|_{L^2(\Omega,\C^2)}^2 \right)^{1/2}.
\]

We set the eigenvalue problem 
\begin{equation}\label{eq:eige1}
 \begin{cases}
  (i\nabla + A_a^\alpha)^2 \varphi = \lambda \varphi &\text{in }\Omega\\
  \varphi=0 &\text{on }\partial \Omega,
 \end{cases}
\end{equation}
in a weak sense, that is $\lambda\in\C$ is an eigenvalue of
problem \eqref{eq:eige1} if there exists $u\in
H^{1,a}_{0}(\Omega,\C)\setminus\{0\}$ (called eigenfunction) such that
\[
\int_\Omega (i\nabla + A_a^\alpha) u 
\cdot \overline{(i \nabla + A_a^\alpha) v} \, dx 
= \lambda \int_\Omega u \overline{v} \,dx 
\quad \text{ for all } v \in H^{1,a}_{0}(\Omega,\C).
\]
From classical spectral theory, 
for every $(a,\alpha) \in \Omega 
\times \R$, the eigenvalue problem \eqref{eq:eige1}
admits a diverging 
sequence of real and positive eigenvalues $\{\lambda_k{(a,\alpha)}\}_{k\geq 1}$ 
with finite multiplicity. 
These eigenvalues 
also have a variational characterization given by 
\begin{equation} \label{eq:var-char}
\lambda_k(a, \alpha) = \min \Big\{ \sup_{u \in W_k \setminus\{0\}} 
\frac{ \int_{\Omega} |(i \nabla + A_a^\alpha) u|^2 }{\int_{\Omega} |u|^2 }
\, : \,  W_k \text{ is a linear $k$-dim subspace of } H^{1,a}_0(\Omega, \C), 
\Big\}.  
\end{equation}

The paper \cite{AbatangeloNys2017} started the study of multiple eigenvalues 
of this operator with respect both to the position of the pole $a\in\Omega$ 
and the circulation $\alpha\in(0,1)$. It shows that 
multiple eigenvalues in general occur, even if 
under suitable assumptions
they are very rare locally with respect to the two parameters.
Here we just mention that these assumptions rely on the local behavior of the corresponding
eigenfunctions. Moreover, to the best of our knowledge, no result is available about this rareness 
globally with respect to the two parameters, yet 
(on this general theme the interested reader can see \cite{Teytel1999}).

As already mentioned, in this paper we consider the eigenvalue problem 
when $\Omega$ is the unit disk
$D:=\{ (x_1,x_2)\in\R^2:\ {x_1}^2 + {x_2}^2 <1\}$
and the circulation $\alpha=\tfrac12$, i.e. the problem 
\begin{equation}\label{eq:eige}
 \begin{cases}
  (i\nabla + A_a)^2 \varphi = \lambda \varphi &\text{in }D\\
  \varphi=0 &\text{on }\partial D.
 \end{cases}
\end{equation}
Throughout the paper we will erase the index $\alpha$, since it is fixed $\alpha=\tfrac12$. 
Because of this choice, 
in view of the correspondance between the magnetic 
problem and a real Laplacian problem on a double covering manifold 
(see \cite{HHOO1999,NT2010}), 
the operator \eqref{eq:operator} behaves as a \emph{real} operator.
As a consequence, 
the nodal set of the eigenfunctions of operator \eqref{eq:operator} 
(i.e.\ the set of points where they vanish) is made of curves and 
not of isolated points as we could expect for complex valued functions. 
More specifically, the magnetic eigenfunctions always have an \emph{odd} 
number of nodal lines ending at the singular point $a$, and therefore 
at least one. 

In particular, we are going to focus our attention on the first eigenvalue to problem \eqref{eq:eige}
and to study its multiplicity as the pole is moving from the origin around the disk. 
One can prove that this situation fulfills the assumptions of 
\cite[Theorem 1.6]{AbatangeloNys2017}, 
so that we know that the origin is {\em locally} the only point where the first eigenvalue is double. 
The main result of the paper is then the following

\begin{theorem}\label{t:main}
 Let $\lambda(a)$ be the first eigenvalue of Problem \eqref{eq:eige}, 
 i.e. $\lambda(a) := \lambda_1(a,\tfrac12)$. 
 It is simple if and only if $a\neq0$. 
\end{theorem}
We recall that the necessary condition is still known (see \cite{BNHHO09}). 
The new result is in fact the sufficient condition.  
The proof relies essentially in two steps. 
Firstly, we observe that 
eigenvalue functions are radial functions. 
Thanks to the local analytic regularity of eigenvalues 
with respect to analytic perturbations of the problem, 
the double eigenvalue for $a=0$ immediately splits in two locally analytic branches, 
which a priori can be the same.  
We will show that in fact they are really different by means 
of their Taylor expansion's first terms. 
The first derivatives of the two branches at the origin can be computed in terms of the
corresponding eigenfunctions' asymptotic expansions 
in the spirit of \cite{AbatangeloNys2017}. This is the content of Section \ref{sec:split}.  

From a technical point of view,
the disk gives us chances 
to compute eigenfunctions explicitly.
This can be done by reducing problem \eqref{eq:eige} to a suitable
weighted Laplace eigenvalue problem on the double covering and thanks to a certain spectral 
equivalence between Problem \ref{eq:eige} and suitable Laplace eigenvalue problems with mixed boundary conditions
(see Section \ref{sec:explicit}). 
This is enough to prove that the first derivatives of the two aforementioned analytic branches 
computed at the origin are different, 
in particular with opposite sign, thus concluding Section \ref{sec:split}.

The proof is concluded in Section \ref{sec:monotone} thanks to the continuity and monotonicity 
of the two branches up to the boundary of the domain.

\subsection{Motivations}

The interest in Aharonov-Bohm operators with half-integer circulation
$\alpha=\tfrac12$ is motivated by the fact that nodal domains of their 
eigenfunctions are strongly related to spectral
minimal partitions of the Dirichlet Laplacian, i.e. partitions of the domain minimizing the largest of the first
eigenvalues on the components, in the special case when they present points of odd
multiplicity (see \cite{BNHHO09}). We refer to papers \cite{BNH11, BNL14,
  H10, HHOO1999, HHO10, HHO13, HHOT09, HHOT10', HHOT10} for details on the deep
relation between behavior of eigenfunctions, their nodal domains, and
spectral minimal partitions.
Related to this, the investigation carried
out in \cite{AbatangeloFelli2015,AbatangeloFelli2016,AbatangeloFelliLena2016,BonnaillieNorisNysTerracini2014,Lena2015,NorisNysTerracini2015} 
highlighted a strong connection
between nodal properties of eigenfunctions and asymptotic expansion of
the function which maps the position of the pole $a$ in the domain to eigenvalues of the operator $(i\nabla +A_a)^2$
(see also \cite[Section 3]{Abatangelo2016} for a brief overview).

The interest in the case of disk comes from the seminal papers \cite{HHO10} and 
\cite{BonnaillieHelffer2013},
where the so-called {\em Mercedes Star Conjecture} is introduced and discussed . 
Roughly speaking, the conjecture evokes that the spectral minimal 3-partition for 
the disk is in fact the {\em Mercedes Star} partition (see \cite[Figure 1]{BonnaillieHelffer2013}). 

For what concerns us, the disk gives us the opportunity to begin to tackle the interesting question 
about how rare multiple eigenvalues are with respect to the position of the pole 
globally in the domain. This is a first contribution to carry on the analysis started in 
\cite{AbatangeloNys2017}. 
On the other hand, 
the present paper is not dealing directly with the aforementioned conjecture, 
but it presents arguments which may be useful towards it.
Finally, Theorem \ref{t:main} validates numerical 
simulations presented in \cite[Figure 1]{BonnaillieHelffer2013} for the first eigenvalue.

\section{Explicit eigenfunctions and eigenvalues}\label{sec:explicit}

The aim of this section is exploiting the symmetry of the disk in order to deduce 
peculiar features of eigenvalues to Problem \eqref{eq:eige}. Firstly, we recall that 
the map $a\mapsto \lambda_k(a)$ is a radial function for any $k\in\N\setminus \{0\}$.

\subsection{Eigenfunctions in the double covering}\label{s:eigedoublecovering}

In the papers \cite[Lemma 3.3]{HHOO1999} and \cite[Section 3]{NT2010}
it is shown that in case of half-integer circulation the considered operator is equivalent to the
standard Laplacian in the double covering.
We then briefly recall  
some basic facts about Aharonov--Bohm operators.
For any 
$a \in \mathbb{R}^2$, we define $\theta_a \, : \, \mathbb{R}^2 \setminus 
\{ a\} \to [0,2\pi)$ the polar angle centered at $a$ such that
\begin{equation} \label{eq:polar-angle-a}
\theta_a(a + r( \cos t, \sin t)) = t, \quad \text{ for } t \in [0, 2 \pi).
\end{equation}
Thus, it results (see \cite{FelliFerreroTerracini2011,HHOO1999,AbatangeloNys2017} for deeper explanations)
that $2 A_a$ is gauge equivalent to $0$, as 
$
2 A_a = - i e^{- i \theta_a} \nabla e^{i \theta_a} = \nabla \theta_a.
$ 
We introduce the following antilinear and antiunitary operator
\[
K_a u = e^{i \theta_a} \overline{u}.
\]
which 
depends on the position of the pole $a\in \Omega$ through the angle $\theta_a$.
It results that $(i \nabla + A_a)^2$ and $K_a$ commute.
The restriction of the scalar product to $L^2_{K_a}(\Omega):= \{ u \in L^2(\Omega,\C) \, : \, K_a u = u \}$ gives 
it the structure of a real Hilbert space and
commutation implies that eigenspaces are stable under the action of $K_a$. Then
we can find a basis of $L^2_{K_a} (\Omega)$ formed by $K_a$-real eigenfunctions of 
$(i \nabla + A_a)^2$.
Being allowed to consider $K_a$-real eigenfunctions of $(i \nabla + A_a)^2$ 
allows to reduce the analysis to the real operator $(i \nabla + A_a)^2_{L^2_{K_a}(\Omega)}$ 
in the real space $L^2_{K_a}(\Omega)$.

\begin{definition}(\cite[Lemma 2.3]{BonnaillieNorisNysTerracini2014}, \cite[Lemma 3.3]{HHOO1999})\label{d:doublecov}
 Let $\Omega\subset \R^2$ be an open simply connected and bounded set. 
 Let $a\in \Omega$ be the pole of the operator. 
 The {\em double covering} of $\Omega$ is the set 
 \[
  \tilde \Omega :=\{ y\in\C:\ y^2+a \in \Omega \}.
 \]
\end{definition}

\begin{lemma}(\cite[Lemma 2.3]{BonnaillieNorisNysTerracini2014})\label{l:bnnt}
 Let $\theta $ denote the angle of the polar coordinates in $\R^2$. 
 If $\varphi$ is a $K_0$-real eigenfunction of the problem \eqref{eq:eige} for $a=0$,
 then the function 
 \[
  \psi(y) := e^{-i\theta(y)} \varphi(y^2) \text{ defined in }\tilde D
 \]
is real valued and it is a solution to the problem
\begin{equation}\label{eq:eigedouble}
\begin{cases}
 -\Delta \psi = 4\lambda |y|^2\, \psi &\text{in }\tilde D\\
 \psi=0 &\text{on }\partial \tilde D.
 \end{cases}
\end{equation}
\end{lemma}

The second basic special feature of the disk is stated in the following

\begin{lemma}\label{l:doubledisk}
 When $a=0$, the double covering of the unit disk $D$ 
 can be identified with the twofold unit disk $D$.
\end{lemma}
\begin{proof}
By Definition \ref{d:doublecov}, the double covering of the unit disk $D$ is
 \[
  \Omega_0:=\{ y\in\C:\ y^2\in D \}. 
 \]
 If we identify $\C$ with $\R^2$ in the standard way and consider the polar coordinates 
 $(x_1,x_2)= \rho(\cos \theta,\sin\theta)$ we need that 
 \[
  (y_1,y_2)= \rho^2(\cos 2\theta,\sin2\theta)\in D.
 \]
 Then, observing that $y_1={x_1}^2 - {x_2}^2$ and $y_2=2x_1x_2$, a simple computation shows that 
 ${y_1}^2 + {y_2}^2 = ({x_1}^2 + {x_2}^2)^2 <1$.
\end{proof}

Thanks to Lemma \ref{l:doubledisk}, we are in position to have an explicit expression 
of eigenfunctions to Problem \eqref{eq:eige} by means of Bessel and trigonometric functions.

\begin{lemma}\label{l:doubleeigenf}
 If $\lambda_0$ is an eigenvalue of the problem \eqref{eq:eigedouble}, then it is double and 
 its 
 eigenfunctions take the form
 \[
 \psi(\rho\cos\theta,\rho\sin\theta)=
 A\,J_{n/2}(\sqrt{\lambda_0}\rho^2) \cos(n\theta)\  +\ 
B\,J_{n/2}(\sqrt{\lambda_0}\rho^2)\sin(n\theta),
\quad y=(\rho\cos\theta,\rho\sin\theta)\in \tilde D
 \]
 with $A,B\in\R$ and for some $n\in\N\setminus\{0\}$. 
 
 Coming back to the original problem \eqref{eq:eige} on the original domain $D$, 
 $\lambda_0$ is a double eigenvalue of the problem \eqref{eq:eige} and its eigenfunctions 
 take the form 
 \begin{equation}\label{eq:coeff2}
  \varphi(r\cos t,r\sin t)= e^{i\frac t2} J_{n/2}(\sqrt{\lambda_0}r) \left(A\,\cos\big(n\frac t2\big) + B\,\sin\big(n\frac t2\big)\right)
 \quad x=(r\cos t,r\sin t)\in D.
 \end{equation}
\end{lemma}
\begin{proof}
Standard separation of variables $\psi(\rho\cos \theta,\rho\sin \theta)= u(\rho)v(\theta)$ leads to
\begin{align*}
 v(\theta) = C 
 \text{ or } v(\theta) = A\cos (n\theta) + B\sin(n\theta) \text{ for }n\in\N 
\end{align*}
being $A,B,C\in\R$.
The radial part produces a Bessel-type equation which reads
\[
{\rho^2} \frac{d^2 u}{d\rho^2} + {\rho}\frac{d u}{d\rho} + (4\lambda_0 \rho^4 - n^2)u(\rho)=0
\]
whose solutions are given by the so-called modified Bessel functions
$ J_{n/2} (\sqrt{\lambda_0}\rho^2)$ or $ J_{-n/2} (\sqrt{\lambda_0}\rho^2)$
(for the modified Bessel functions, see the book by Watson \cite{Watson}).
From the results in \cite{FelliFerreroTerracini2011,HHOO1999} we know that the eigenfunction 
is regular at the origin, so its radial part will be given in terms of the only $J_{n/2}$.
Imposing the boundary conditions at $\rho=1$, we find $J_{n/2}(\sqrt{\lambda_0})=0$, which means that 
\[
 \lambda_0 = {\alpha_{n/2,k}}^2 \quad \text{for some }k\in \N,
\]
where $\{\alpha_{n/2,k}\}_{k\in\N}$ denote the sequence of zeros of the Bessel function $J_{n/2}$.
This concludes the first part of the statement. 
By virtue of Lemma \ref{l:bnnt} the rest of the statement follows.
\end{proof}

Note that the case of the disk is covered by the paper \cite{BNHHO09}: the fact that every eigenvalue is
double was already provided by \cite[Proposition 5.3]{BNHHO09} in a more general context. 
Nevertheless, this is not the main point we are interested in.

We recall that there is a connection between the zeros of the Bessel functions (to this aim we refer to 
\cite[Chapter XV]{Watson}): in particular, the positive zeros of the Bessel function $J_{\frac n2}$
are interlaced with those of the Bessel function $J_{\frac{n+1}2}$ and by Porter's Theorem 
the positive zeros of $J_{\frac n2}$
are interlaced with those of the Bessel function $J_{\frac{n+2}2}$.
Then, denoting $z_{\frac n2,k}$ the $k$-th zero of the Bessel function $J_{\frac n2}$, we have
\[
 0 < z_{\frac12,1} < z_{\frac32,1} < z_{\frac52,1} < z_{\frac12,2} < z_{\frac72,1} < \ldots
\]

\begin{remark}\label{r:ms}
The first case is then $(n,k)=(1,1)$ and 
it corresponds to the double first eigenvalue for the Aharonov--Bohm
operator with half-integer circulation and pole at the origin. 

The second case is $n=3$ and $k=1$, which produces the double third eigenvalue. 
\end{remark}

\subsection{Isospectrality and consequences on eigenvalues}
We introduce two auxiliary problems. Let us denote 
$D^+:= \{ (x_1,x_2)\in D:\ x_2>0 \}$. 
\begin{definition}(\cite{Lena2015})\label{d:auxiliary}
 The two problems
 \begin{align}\label{eq:dnnd}
  \begin{cases}
   -\Delta u = \lambda u &\text{in }D^+\\
   u=0 &\text{on }\partial D^+ \setminus (t,1]\times\{0\}\\
   \frac{\partial u}{\partial \nu} =0 &\text{on }(t,1]\times\{0\}
  \end{cases}
  \quad 
  \begin{cases}
   -\Delta u = \lambda u &\text{in }D^+\\
   u=0 &\text{on }\partial D^+ \setminus [-1,t)\times\{0\} \\
   \frac{\partial u}{\partial \nu} =0 &\text{on }[-1,t)\times\{0\} 
  \end{cases}
 \end{align}
 are called {\em Dirichlet--Neumann} and {\em Neumann--Dirichlet}
 eigenvalue problem for the Laplacian in the upper half-disk, respectively. 
\end{definition}
We recall the following result proved in \cite{BNHHO09} (see also \cite[Proposition 5.3]{Lena2015}).
\begin{lemma}(\cite{BNHHO09})\label{l:isospectrality}
Let $a=(t,0)$ for $t\in[0,1]$. 
 The set of the eigenvalues of Problem 
 \eqref{eq:eige} $\{\lambda_j(t)\}_{j\geq1}$ is the union (counted with multiplicity) of the 
 sequences $\{\lambda_j^{DN}(t)\}_{j\geq1}$ and $\{\lambda_j^{ND}(t)\}_{j\geq1}$, 
 being $\{\lambda_j^{DN}(t)\}_{j\geq1}$ and $\{\lambda_j^{ND}(t)\}_{j\geq1}$ the set of the eigenvalues 
 of the Dirichlet--Neumann and Neumann--Dirichlet problems \eqref{eq:dnnd} respectively. 
\end{lemma}

By virtue of the latter Lemma \ref{l:isospectrality} and the continuity result stated in 
\cite{Lena2015} for Aharonov--Bohm eigenvalues
(see also \cite[Section 10]{DaugeHelffer1993}),
the following result holds true.
\begin{lemma}(\cite{Lena2015}, \cite{DaugeHelffer1993})\label{l:continuity}
Fix $k\in\N\setminus\{0\}$ and denote 
$\lambda_k^{DN}(t)$ ($\lambda_k^{ND}(t)$) 
the $k$-th eigenvalue of the Dirichlet--Neumann problem in \eqref{eq:dnnd} (Neumann--Dirichlet problem, respectively). Then the maps 
\[
 t \mapsto \lambda^{DN}_k(t) \quad t \mapsto \lambda^{ND}_k(t)
 \qquad \text{are continuous in }(-1,1). 
\]
\end{lemma}

We observe that in this case the standard Courant--Fisher characterization of eigenvalues establishes
\begin{equation}\label{eq:courantfisher}
 \lambda_k^{DN}(t) = \min_{\stackrel{E\subset\mathcal H_t \text{ subspace}}{\mathrm{dim}E=k}}\ 
 \max_{u\in E \setminus\{0\}}\frac{\int_{\Omega} |\nabla u|^2}{\int_{\Omega} u^2},
\end{equation}
where
\[
\mathcal H_t :=\left\{ u\in H^1(\Omega):\ u=0 \text{ on }\partial \Omega\setminus (t,1)\times \{0\} 
\text{ and }\frac{\partial u}{\partial\nu}=0 \text{ on } (t,1)\times \{0\} \right\},
\]
analogously for $\lambda_k^{ND}(t)$.

\begin{remark}\label{r:monotonicity}
By 
\eqref{eq:courantfisher},
if $-1<t_1 \leq t_2<1$ then $\mathcal H_{t_2}\subseteq \mathcal H_{t_1}$ and then 
$\lambda_j^{DN}(t_2)\geq\lambda_j^{DN}(t_1)$ for any $j\geq1$,
i.e. the function $t\mapsto \lambda_j^{DN}(t)$ is monotone non-decreasing for any $j\geq1$. 
As well, the function $t\mapsto \lambda_j^{ND}(t)$ is monotone non-increasing for any $j\geq1$.

In the case of the disk, 
one can even see it by noting that $\lambda_j^{DN}(t) = \lambda_j^{ND}(-t)$ because of the symmetry of the disk. 
\end{remark}

Another consequence of Lemma \ref{l:isospectrality} is the following result.
\begin{lemma}\label{l:t=1}
Let us consider the problems in \eqref{eq:dnnd}. For $t=1$ we have
$\lambda_1^{DN}(1) = \lambda_2^{ND}(1)$.
\end{lemma}
We note the latter result can be proved by direct computation, in terms of Bessel-type functions, 
as in the proof of Lemma \ref{l:doubleeigenf}.

Now, if $a=(t,0)$ let us denote $\lambda_j(t)$ the $j$-th eigenvalue of the problem \eqref{eq:eige}.
By Lemma \ref{l:isospectrality}, symmetry of the disk
and Remark \ref{r:monotonicity} (non-increasing monotonicity of the map $t\mapsto \lambda_1^{ND}(t)$), we have
\begin{equation}\label{eq:lambda1}
 \lambda_1(t) = \min\left \{ \lambda_1^{DN}(t),\ \lambda_1^{ND}(t) \right\} 
 = \min\left \{ \lambda_1^{ND}(-t),\ \lambda_1^{ND}(t) \right\} 
 = \lambda_1^{ND}(t) \quad \text{for any }t\in [0,1).
\end{equation}
We have as well
\begin{equation}\label{eq:lambda2}
 \lambda_2(t) = \min\left \{ \lambda_1^{DN}(t),\ \lambda_2^{ND}(t) \right\} = \lambda_1^{DN}(t) \quad \text{for any }t\in [0,1),
\end{equation}
where the last equivalence follows from Lemma \ref{l:continuity}, Remark \ref{r:monotonicity} and Lemma \ref{l:t=1},
recalling that $\lambda_2^{ND}(0)=\lambda_2^{DN}(0)>\lambda_1^{DN}(0)=\lambda_1^{ND}(0)$. 
Indeed, if by contradiction there exists $\bar t \in(0,1)$ such that $\lambda_2^{ND}(\bar t)< \lambda_1^{DN}(\bar t)$, 
then Remark \ref{r:monotonicity} implies 
$\lambda_2^{ND}(1)\leq \lambda_2^{ND}(\bar t)< \lambda_1^{DN}(\bar t)\leq \lambda_1^{DN}(1)$ which denies 
Lemma \ref{l:t=1}.

\section{Immediate splitting of the eigenvalue}\label{sec:split}

The aim of this section is to show that as the pole is moved, 
then the double eigenvalue split and produce two locally {\em different} analytic branches of eigenvalues. 
The first one is stricly monotone decreasing 
whereas the second one is stricly monotone increasing 
in a small neighborhood of the origin,
with respect to the distance of the pole from the origin.
In order to do this, we are going to exploit the results achieved in Section \ref{sec:explicit}.
In addition, by rotational symmetry, we will restrict ourselves to the case 
when the pole is moving along $x_1$-axis.

\subsection{ Analytic perturbation with respect to the pole}\label{sec:analytic}

As already pointed out in the Introduction (see also 
\cite[Section 2]{AbatangeloNys2017}, \cite{Lena2015}), 
as the pole moves not only the operator changes, but also this produces 
different variational settings: functional spaces depend on the position of the pole. 
In order to study the moving pole's effect on eigenvalues, 
first of all we need to define a family of diffeomorphisms which allow us 
to set the eigenvalue problem on a fixed domain, in the spirit of \cite{AbatangeloNys2017,Lena2015}.

We consider a particular case 
of the local perturbation introduced in \cite[Subsection 5.1]{AbatangeloNys2017}. 
Let $\xi \in C^{\infty}_c(\R^2)$ be a cut-off function such that
\begin{equation}\label{eq:xi}
0 \leq \xi \leq 1,\quad \xi \equiv 1 
\text{ on } D_{1/4}(0), 
\quad \xi \equiv 0 
\text{ on } \R^2\setminus D_{1/2}(0), 
\quad |\nabla \xi| \leq 16 
\text{ on }\R^2.
\end{equation}
For $a \in D_{1/4}(0)$, we define the local transformation 
$\Phi_a \in C^\infty (\R^2, \R^2)$ by 
\begin{equation}\label{eq:Phi_a}
\Phi_a (x) = x + a \xi(x).
\end{equation}
Notice that $\Phi_a(0) = a$ and that $\Phi_a'$ is a perturbation 
of the identity 
\[
\Phi_a' = I + a \otimes \nabla\xi = 
\begin{pmatrix}
1 + a_1 \frac{\partial \xi}{\partial x_1} 
& a_1 \frac{\partial \xi}{\partial x_2} \\
a_2 \frac{\partial \xi}{\partial x_1} 
& 1 + a_2 \frac{\partial \xi}{\partial x_2}
\end{pmatrix},
\] 
so that 
\begin{equation}\label{eq:det}
J_a(x):=\det(\Phi_a')= 1 + a_1 \frac{\partial \xi}{\partial x_1} 
+ a_2 \frac{\partial \xi}{\partial x_2}
= 1 + a \cdot \nabla \xi.
\end{equation}
Let $R = 1/128$. Then, if $a \in D_R(0)$, $\Phi_a$ 
is invertible, its inverse $\Phi_a^{-1}$ is also $C^\infty(\R^2, 
\R^2)$, see e.g. \cite[Lemma 1]{Micheletti1972}. 
Then, as in \cite[Section 7]{AbatangeloNys2017}, 
we define $\gamma_a \, : \, L^2(\Omega,\C) 
\to L^2(\Omega,\C)$ by
\begin{equation}\label{eq:gamma_a}
\gamma_a (u) = \sqrt{J_a} (u \circ \Phi_a),
\end{equation}
where $J_a$ is defined in \eqref{eq:det}. Such a transformation $\gamma_a$ 
defines an isomorphism preserving the scalar product in $L^2(\Omega,\C)$. 
Moreover, since $\Phi_a$ and $\sqrt{J_a}$ are $C^\infty$, $\gamma_a$ defines an algebraic 
and topological isomorphism of $H^{1,a}_0(\Omega,\C)$ in $H^{1,0}_0(\Omega,\C)$ 
and inversely with $\gamma_a^{-1}$, see \cite[Lemma 2]{Micheletti1972}.
We notice that $\gamma_a^{-1}$ writes
\[
\gamma_a^{-1} (u) = \left(\sqrt{J_a \circ \Phi_a^{-1}} \right)^{-1} 
( u \circ \Phi_a^{-1}).
\]

With a little abuse of notation we define the application $\gamma_a \, : \, 
( H^{1,a}_0(\Omega,\C) )^\star \to (H^{1,0}_0(\Omega,\C))^\star$ in 
such a way that 
\begin{equation} \label{eq:gamma-a-dual}
\phantom{a}_{( H^{1,0}_0(\Omega,\C) )^\star }\langle 
\gamma_a(f), v \rangle_{H^{1,0}_0(\Omega,\C)} 
= \phantom{a}_{( H^{1,0}_0(\Omega,\C) )^\star }\langle 
f, \gamma_a^{-1}(v)
\rangle_{H^{1,a}_0(\Omega,\C)},
\end{equation}
for any $f \in ( H^{1,a}_0(\Omega,\C) )^\star$, and inversely for 
$\gamma_a^{-1} \, : \, ( H^{1,0}_0(\Omega,\C) )^\star \to (H^{1,a}_0(\Omega,\C))^\star$.

We define the new operator $G_{a} 
\, : \, H^{1,0}_0(\Omega,\C) \to (H^{1,0}_0(\Omega,\C))^\star$ 
by the following relation
\begin{equation} \label{eq:operator-G-a}
G_{a} \circ \gamma_a = \gamma_a \circ (i \nabla + A_a)^2,
\end{equation}
being $\gamma_a$ defined in \eqref{eq:gamma_a} and \eqref{eq:gamma-a-dual}.
By \cite[Lemma 3]{Micheletti1972} the domain of definition of 
$G_{a}$ is given by $\gamma_a(H^{1,a}_0(\Omega, \C))$, 
it coincides with $H^{1,0}_0(\Omega,\C)$. Moreover, $G_{a}$ 
and $(i \nabla + A_a^\alpha)^2$ are \emph{spectrally equivalent}, 
in particular they have the same eigenvalues with the same multiplicity 
and the map
$a  \mapsto G_{a}$ is $C^\infty (D_R(0), BL( H^{1,0}_0(\Omega,\C), (H^{1,0}_0(\Omega,\C))^\star )$.

Now, let us consider the special case $a=(a_1,0)$, which 
means moving the pole just along the $x_1$-axis. 
For simplicity, in the following we denote
\[
 t:=a_1 \quad \text{and} \quad G_t:=G_{(a_1,0)}. 
\]
Then, following the same argument in \cite[Section 4]{Lena2015}, 
the family $t\mapsto G_{t}$ is an {\em analytic family of type (B) in the sense of Kato} 
with respect to the variable $t$. 
In order to prove it, by definition (see \cite[Chapter 7, Section 4]{Kato}) 
we need to show that the quadratic form $\mathfrak g_t$ associated to $G_t$, defined as
\[
 \mathfrak g_t(u)=\phantom{a}_{( H^{1,0}_0(\Omega) )^\star }\langle 
 G_t u,u
 \rangle_{H^{1,0}_0(\Omega)},
\]
is an {\em analytic family of type (a)
in the sense of Kato}, i.e. it fulfills the following two conditions: 
\begin{itemize}
 \item[(i)] the form domain is independent of $t$;
 \item[(ii)] the form $\mathfrak g_t(u)$ is analytic 
 with respect to the parameter $t$ for any $u$ in the form domain.
\end{itemize}
The first assertion follows from \eqref{eq:operator-G-a} (see \cite[Section 7.1]{AbatangeloNys2017}),
whereas the second one follows from \cite[Lemmas 5.1,5.2,7.1]{AbatangeloNys2017}
possibly shrinking the interval $(-R,R)$ where the parameter $t$ is varying.
%
The Kato-Rellich perturbation theory gives some information in the case when 
the considered eigenvalue is not
simple. 
Let $\lambda_0$ be any double eigenvalue of $G_0$. 
Then 
there exist a family of $2$ linearly independent 
$L^2(\Omega)$-normalized eigenfunctions $\{u_j (t)\}_{j=1,2}$ relative to the associated eigenvalue  
$\mu_j(t)$ for $j=1,2$ which depend analytically on the parameter $t$ and
such that for $j=1,2$ $\mu_j (0) = \lambda_0$ and 
$\mu_j (t)$ is an eigenvalue of the operator
$G_t$. 
We recall that $G_t$ has the same eigenvalues with the same multiplicity as operator $(i\nabla+A_{(t,0)})^2$.
Note that the 
$2$ functions $t \mapsto \mu_1(t),\ t\mapsto \mu_2 (t)$ are not {\em a priori} necessarily distinct. 
The Feynman-Hellmann formula (see \cite[Chapter VII, Section 3]{Kato}) then tells us that
\begin{equation}\label{eq:FHformula}
 \mu_j'(0) = \phantom{a}_{(H^{1,0}_0(\Omega,\C))^\star} \left\langle G'(0)[t]\,u_j(0), u_j(0) \right \rangle_{H^{1,0}_0(\Omega,\C)}.
\end{equation}

\subsection{Computing the derivative at $0$ of the two branches}

The aim of this subsection is showing that the two (\textit{a priori} not necessarily different) 
analytic branches $t\mapsto \mu_j(t)$, $j=1,2$, have a different derivative at $t=0$. 
In order to do this, we refer to the paper \cite{AbatangeloNys2017}. In particular, for $j=1,2$
\eqref{eq:FHformula} together with \cite[Lemma 8.2, Lemma 8.6]{AbatangeloNys2017} yield
\begin{equation}\label{eq:FHformula2}
 \mu_j'(0) = \phantom{a}_{(H^{1,0}_0(\Omega,\C))^\star} \left\langle
 G'(0)[t]\,u_j(0), u_j(0) 
 \right \rangle_{H^{1,0}_0(\Omega,\C)}
 = \frac\pi2 ({A_j}^2 - {B_j}^2) 
\end{equation}
where $A_j,B_j\in\R$ are the coefficients in the expansion \eqref{eq:coeff2}.

What is left is detecting $u_j(0)$ for $j=1,2$. To this aim, we are going to exploit the 
symmetry property of the domain with respect to the $x_1$-axis.
We refer to \cite{BNHHO09} and define the antiunitary antilinear operator $\Sigma:\ L^2 (D)\to L^2 (D)$
\[
\Sigma u := \bar u \circ \sigma,  
\]
being $\sigma(x_1,x_2)=(x_1,-x_2)$. 
We have that $\Sigma$ and $(i\nabla +A_0)^2$ commute (see \cite[Section 5]{BNHHO09}),
as well as $\Sigma $ and $K_0$. This means $L^2_{K_0}$ is stable under the action of $\Sigma$.
Thus, if we write 
\begin{equation*}L^2_{K,\Sigma}(\Omega):= L^2_K(\Omega)\cap \mbox{ker}(\Sigma-Id)
\qquad
L^2_{K,a\Sigma}(\Omega):= L^2_K(\Omega)\cap \mbox{ker}(\Sigma+Id),\end{equation*}
then we
have the orthogonal decomposition
\begin{equation}\label{eq:orthogonal}
L^2_K(\Omega)=L^2_{K,\Sigma}(\Omega)\oplus L^2_{K,a\Sigma}(\Omega).
\end{equation}
We can
therefore define the operators $(i\nabla +A_0)^2_{\Sigma}$ and $(i\nabla +A_0)^2_{a\Sigma}$,
restrictions of $(i\nabla +A_0)^2$ to $L^2_{K,\Sigma}(\Omega)$ and
$L^2_{K,a\Sigma}(\Omega)$ respectively. The spectrum of $(i\nabla +A_0)^2$ is the
union (counted with multiplicities) of the spectra of $(i\nabla +A_0)^2_{\Sigma}$
and $(i\nabla +A_0)^2_{a\Sigma}$. Lemma \ref{l:isospectrality} is then completed by the following result.
\begin{lemma}(\cite[Propositions 5.7 and 5.8]{BNHHO09})\label{l:isospectrality_eigenf}
 If $u$ is a $K_0$-real $\Sigma$-invariant eigenfunction of $(i\nabla +A_0)^2$ then the restriction to $D^+$
 of $e^{-\tfrac i2 \theta_0}u$ is a real eigenfunction of the Dirichlet--Neumann problem in \eqref{eq:dnnd}. 
  If $u$ is a $K_0$-real $a\Sigma$-invariant eigenfunction of $(i\nabla +A_0)^2$ then the restriction to $D^+$
 of $e^{-\tfrac i2 \theta_0}u$ is a real eigenfunction of the Neumann--Dirichlet problem in \eqref{eq:dnnd}.
 Conversely, if $v$ is an eigenfunction of the Dirichlet--Neumann problem in $D^+$, if 
 $\tilde v$ is the even extension of $u$ in $D$,
 the function $e^{\tfrac i2 \theta_0}\tilde v$ is a ($K_0$-real) $\Sigma$-invariant eigenfunction of $(i\nabla +A_0)^2$. 
 If $v$ is an eigenfunction of the Neumann--Dirichlet problem in $D^+$, if 
 $\tilde v$ is the odd extension of $u$ in $D$,
 the function $e^{\tfrac i2 \theta_0}\tilde v$ is a ($K_0$-real) $a\Sigma$-invariant eigenfunction of $(i\nabla +A_0)^2$. 
\end{lemma}
In view of \eqref{eq:gamma_a} and \eqref{eq:operator-G-a} we have  
that $u_1(0)$ and $u_2(0)$ are two $K_0$-real linearly independent eigenfunctions of $(i\nabla+A_0)^2$.
Therefore via \eqref{eq:orthogonal}, Lemma \ref{l:isospectrality_eigenf} and Lemma \ref{l:doubleeigenf}
$u_1(0)$ is $a\Sigma$-invariant whereas 
$u_2(0)$ is $\Sigma$-invariant.
From Lemma \ref{l:doubleeigenf}, Remark \ref{r:ms} and the asymptotic expansion of the Bessel functions 
(see e.g. \cite[Chapter 3]{Watson}) there exist $A,B\in\R\setminus \{0\}$ such that
 \begin{align}
  u_1(r\cos t,r\sin t)=  e^{i \frac{t}{2}} r^{1/2} 
\,B \sin \frac{t}{2} + O(r^{\tfrac32}) \quad \text{ as } r \to 0^+ \\
u_2(r (\cos t, \sin t)) =   e^{i \frac{t}{2}} r^{1/2} 
\,A \cos \frac{t}{2} + O(r^{\tfrac32}) \quad \text{ as } r \to 0^+.
 \end{align}
Equations \eqref{eq:FHformula} and \eqref{eq:FHformula2} immediately give
\begin{align}
 &\mu_1'(0) = - \frac\pi2 \,{B}^2 \ <\ 0,\\
 &\mu_2'(0) =  \frac\pi2 \,{A}^2 \ >\ 0,
\end{align}
thus concluding the first step towards our main result.

\section{Conclusion}\label{sec:monotone}

We are now in position to conclude the proof of our main result. 
\begin{proof}[Proof of Theorem \ref{t:main}]
Thanks to rotational invariance of eigenvalues, it is sufficient to prove that if  
$a=(t,0)$ and $\lambda_1(t)$ is the first eigenvalue of the problem \eqref{eq:eige}, 
which is double for $t=0$, then $\lambda_1(t)$ is simple for any $t\in(0,1)$. 

By the results of Section \ref{sec:split}, there exists $\delta>0$ 
such that the two 
analytic eigenbranches $\mu_1(t)$ and $\mu_2(t)$ are different for $t\in(-\delta,\delta)$, since 
\begin{equation}\label{eq:derivate}
\mu_1'(0)<0 \quad \text{whereas} \quad  \mu_2'(0)>0. 
\end{equation}
Moreover, we have that 
\begin{equation}\label{eq:rami}
 \lambda_1(t) = 
 \begin{cases}
  \mu_2(t) &\text{ for }t\in(-\delta,0]\\
  \mu_1(t) &\text{ for }t\in[0,\delta),
 \end{cases}
\end{equation}
since $\mu_j(t)$ are eigenvalues of the operator $G_t$ which is spectral equivalent to $(i\nabla +A_a)^2$
with $a=(t,0)$.
In order to prove that it is simple for $t\in(0,1)$, it will be sufficient to prove that 
$\lambda_1(t)<\lambda_2(t)$ for $t\in(0,1)$. 
This is guaranteed by \eqref{eq:rami}, \eqref{eq:derivate},\eqref{eq:lambda1}, \eqref{eq:lambda2} and Remark \ref{r:monotonicity}.
This concludes the proof of Theorem \ref{t:main}. 
\end{proof}

\begin{figure}[h]
\begin{center}
\psset{unit=0.5cm}
\begin{pspicture}(-10,0)(10,10)
\psaxes*[labels=none,ticks=none]{->}(0,0)(-10,0)(10,10)
\uput[270](10,0){$t$}
\uput[270](0,5){$\lambda_1^{ND}(0)=\lambda_1^{DN}(0)$}
\uput[0](3.5,6){$\lambda_{1}^{DN}(t)$}
\uput[0](-6.5,6){$\lambda_{1}^{ND}(t)$}
\uput[0](4.3,1.5){$\lambda_{1}^{ND}(t)$}
\uput[0](-5.3,1.5){$\lambda_{1}^{DN}(t)$}
\psline[linewidth=.5pt,linestyle=dashed](8,0)(8,7)
\psline[linewidth=.5pt,linestyle=dashed](-8,0)(-8,7)
\psline[linewidth=.5pt,linestyle=dashed](-8,7)(8,7)
\psdot[dotstyle=*](0,4.4)
\uput[270](8,0){$1$}
\uput[270](-8,0){$-1$}
\uput[270](0,0){$0$}
\pscurve(-8,0.8)(-7,1)(6,6.9)(7,7)
\pscurve(-7,7)(-6,6.9)(7,1)(8,0.8)
\pscurve(0,9)(0.8,9)(6,7)(7,7)
\psline(7,7)(8,7)
\pscurve(-7,7)(-6,7)(-0.8,9)(0,9)
\psline(-7,7)(-8,7)
\psdot[dotstyle=*](0,9)
\uput[90](0,8.6){$\lambda_2^{ND}(0)=\lambda_2^{DN}(0)$}
\uput[0](3.5,8){$\lambda_{2}^{ND}(t)$}
\uput[0](-7.5,8){$\lambda_{2}^{DN}(t)$}
\end{pspicture}
\end{center}
\caption{The double first Aharonov--Bohm eigenvalue $\lambda_1(0)$ splits in two different branches of simple eigenvalues 
up to the boundary.}\label{fig:1}
\end{figure}
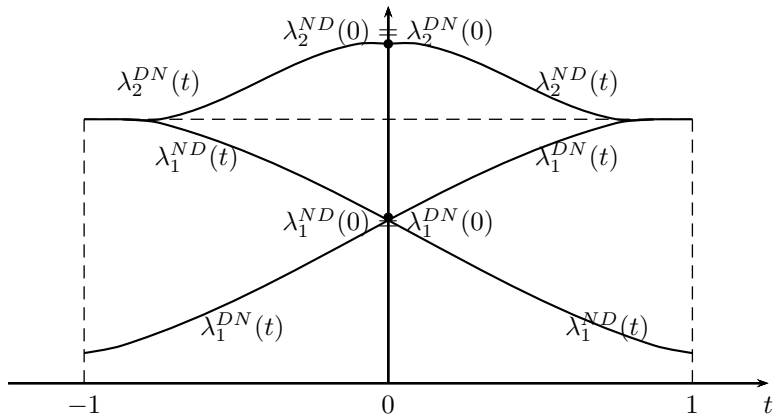

\section*{Acknowledgements}

The author would like to thank dr. Manon Nys for many discussions,
as well as the anonymous referee of the paper \cite{AbatangeloNys2017} 
who excited our interest about the theme. 

The author is
partially supported by the project ERC Advanced Grant
2013 n. 339958: ``Complex Patterns for Strongly Interacting Dynamical
Systems -- COMPAT'', by the PRIN2015 grant ``Variational methods, with
applications to problems in mathematical physics and geometry''
and by the 2017-GNAMPA project ``Stabilità e analisi
  spettrale per problemi alle derivate parziali''.

\end{document}